\documentclass[a4paper]{article}

\usepackage[utf8]{inputenc}
\usepackage[T1]{fontenc}

\usepackage{amsmath}
\usepackage{amssymb}
\usepackage{hyperref}

\usepackage{amssymb}
\usepackage{amsmath}
\usepackage{mathtools}
\usepackage{mathrsfs}
\usepackage{textcomp}
\usepackage{hhline}
\usepackage{diagbox}
\usepackage{mathdots,enumerate}
\usepackage{graphicx}
\usepackage{tikz}
\usepackage{graphicx}      

\newcommand{\Dc}{\mathcal{D}}
\newcommand{\Lc}{\mathcal{L}}
\newcommand{\Vs}{\mathcal{V}_{\rm sys}}

\DeclareMathOperator{\re}{Re}

\DeclareMathOperator{\diag}{diag}
\DeclareMathOperator{\ran}{ran}

\newcommand{\R}{\mathbb{R}}
\newcommand{\C}{\mathbb{C}}
\newcommand{\N}{\mathbb{N}}

\newcommand{\ddt}{\tfrac{{\rm d}}{{\rm d}t}}

\newenvironment{smallpmatrix}
{\left(\begin{smallmatrix}}
{\end{smallmatrix}\right)}

\newenvironment{smallbmatrix}
{\left[\begin{smallmatrix}}
{\end{smallmatrix}\right]}

\usepackage{pgf,tikz,pgfplots}
\pgfplotsset{compat=1.14}
\usepackage{mathrsfs}
\usetikzlibrary{arrows}

\usepackage[amsmath,amsthm,thmmarks]{ntheorem}

\theoremstyle{plain}
\newtheorem{definition}{Definition}[section]

\newtheorem{prop}[definition]{Proposition}
\newtheorem{lem}[definition]{Lemma}
\newtheorem{corollary}[definition]{Corollary}

\newtheorem{rem}[definition]{Remark}

\title{On the stability of port-Hamiltonian descriptor systems
 \thanks{H.G. acknowledges the support by the Deutsche Forschungsgemeinschaft (DFG) within the Research Training Group GRK 2583 "Modeling, Simulation and Optimization of Fluid Dynamic Applications”.}
        }

\author{Hannes Gernandt\thanks{Institut für Mathematik, TU Berlin, Stra\ss e des 17. Juni 136, 10587 Berlin, Germany ({\tt gernandt@math.tu-berlin.de}).}
        \and Fr\'{e}d\'{e}ric E. Haller \thanks{Universit\"{a}t Hamburg, Bundesstraße 55, 20146 Hamburg ({\tt  frederic.haller@uni-hamburg.de}).}}

\begin{document}

\maketitle

\begin{abstract}
We characterize stable differential-algebraic equations (DAEs) using a generalized Lyapunov inequality. The solution of this inequality is then used to rewrite stable DAEs as dissipative Hamiltonian (dH) DAEs  on the subspace where the solutions evolve. Conversely, we give sufficient conditions guaranteeing stability of dH DAEs. Further, for stabilizable descriptor systems we construct solutions of generalized algebraic Bernoulli equations which can then be used to rewrite these systems as pH descriptor systems. Furthermore, we show how to describe the stable and stabilizable systems using Dirac and Lagrange structures.
\end{abstract}







\section{Introduction}

In this note, we consider generalized port-Hamiltonian systems (pH systems) of the form
\begin{align}
\label{ph_sys}
\begin{split}
    \tfrac{\rm d}{{\rm d} t} Ex(t)&=(J-R)Qx(t)+(B-P)u(t),\\
    y(t)&=(B+P)^TQx(t)+(S-N)u(t),
\end{split}
\end{align}
where $E,J,R,Q\in\R^{n\times n}$, $B,P\in\R^{n\times k}$ and $S,N\in\R^{k\times k}$ 
and 
\begin{align}
\label{ph_cond}
\begin{smallbmatrix}
R & P\\P^T&S
\end{smallbmatrix}\geq 0,\ \begin{smallbmatrix}
J & B\\-B^T&N
\end{smallbmatrix}=-\begin{smallbmatrix}
J & B\\-B^T&N
\end{smallbmatrix}^T,\ Q^TE=E^TQ,
\end{align}
see e.g.\ \cite{MehrMora19}. In comparison to classical port-Hamiltonian systems where the coefficient $E$ is the identity, here it might be singular. 

If in \eqref{ph_sys} neither inputs nor outputs are present then the system reduces to the differential-algebraic equation (DAE)
\begin{align}
\label{ph_DAE_i}
\tfrac{\rm d}{{\rm d} t} Ex(t)=(J-R)Qx(t)
\end{align}
where the conditions \eqref{ph_cond} simplify to 
\begin{equation}\label{eq:JRQE}
R\geq 0,\quad J=-J^T, \quad  Q^TE=E^TQ.
\end{equation}
If we further assume $Q^TE\geq0$, these equations are called \emph{dissipative Hamiltonian DAEs} (dH~DAEs), see \cite{MehlMehrWojt20}.

The first aim of this paper is to show that stable DAEs
\begin{align}
\label{dae}
\ddt Ex(t)=Ax(t)
\end{align}
can be rewritten as dH~DAEs \eqref{ph_DAE_i} where the matrix $Q$ is given by the solution of a generalized Lyapunov inequality. 
For this rewriting we have to restrict the coefficients of \eqref{dae} to the smallest subspace where the solutions evolve, also called the \emph{system space}, see \cite{ReisRend15}. Furthermore, it is shown that such a reformulation as pH systems is also possible for stabilizable systems 
\[\ddt Ex(t)=Ax(t)+Bu(t)\]
by introducing a suitable output and using the solutions of generalized algebraic Bernoulli equations. 

The second aim of the paper is to study the stability of DAEs \eqref{dae} which are given in the geometric pH framework from \cite{MascvdSc18}. This more general framework contains dH DAEs \eqref{ph_DAE_i} which was recently shown in \cite{GernHall21}. Here it is shown how systems of the form \eqref{ph_sys} can be embedded into the more general geometric pH framework from \cite{MascvdSc18}.


The outline of the paper is as follows. After recalling some terminologies of matrix pencils in Section~\ref{sec:pre}, we study the behavioral stability of DAEs in Section~\ref{sec:stab}. We characterize the stability of DAEs using a generalized Lyapunov inequality and use their solutions to rewrite stable DAEs as dH DAEs. Further, we prove that dH DAEs  which fulfill $\ker Q\subseteq\ker E$ are stable. In Section~\ref{sec:stabi} we rewrite behaviorally stabilizable DAEs as pH systems~\eqref{ph_sys}. 
Finally, in Section~\ref{sec:geo} we embedded dH DAEs \eqref{ph_DAE_i} and pH systems~\eqref{ph_sys} into the geometric framework from \cite{MascvdSc18}.

\section{Preliminaries} 
\label{sec:pre}
Matrix pencils $sE-A$ are linear matrix polynomials with coefficients in $E,A\in\R^{n\times n}$. We briefly write $sE-A\in\R[s]^{n\times n}$ and this pencil is called \emph{regular} if $\det(sE-A)\in\R[s]$ is non-zero. The \emph{spectrum} $\sigma(E,A)$ of a matrix pencil $sE-A$ is the set of all complex numbers $\lambda$ for which $\lambda E-A$ is not invertible. An eigenvalue is called simple if $\ker(\lambda E-A)$ has dimension one and semi-simple if the dimension of this subspace  coincides with the multiplicity of $\lambda$ as a root of $\det(sE-A)$. 

Recall that every pencil $sE-A\in\R[s]^{n\times m}$ can be transformed to quasi Kronecker form, see e.g. \cite{BergTren12}, i.e.,\ there exists invertible $S\in\R^{n\times n}$ and $T\in\R^{m\times m}$ such that $S(sE-A)T$ is block diagonal with the following four types of blocks 
\begin{subequations}
\label{kcf}
\begin{align}
\label{kcf_a}
sI_{n_0}-A_0&\in\R[s]^{n_
0\times n_0}, 
\end{align}
\begin{align}
sN_{\alpha_i}-I_{\alpha_i}&=\begin{smallbmatrix}
-1&s&&
\\&\ddots&\ddots&
\\&&\ddots&~~s
\\&&&-1
\end{smallbmatrix}\in\R[s]^{\alpha_i\times \alpha_i},\\
sK_{\beta_i}-L_{\beta_i}&=\begin{smallbmatrix}
s&-1&&
\\&\ddots&\ddots&
\\&&s&-1
\end{smallbmatrix}\in\R[s]^{(\beta_i-1)\times \beta_i},\\
\quad  sK_{\gamma_i}^T-L_{\gamma_i}^T&\in\R[s]^{\gamma_i\times (\gamma_i-1)}.
\end{align}
\end{subequations}
The above indices are collected in multi-indices $\alpha=(\alpha_1,\ldots,\alpha_{\ell_{\alpha}})\in\N^{\ell_\alpha}$, $\beta=(\beta_1,\ldots,\beta_{\ell_{\beta}})\in\N^{\ell_\beta}$,  $\gamma=(\gamma_1,\ldots,\gamma_{\ell_{\gamma}})\in\N^{\ell_\gamma}$. For a multi-index $\delta\in\N^{\ell_\delta}$ recall $|\delta|=\sum_{i=1}^{\ell_\delta}\delta_i$. 
Note that for $\beta_i=1$ and $\gamma_i=1$ one adds a zero column and a zero row to the pencil, respectively. The largest $\alpha_i$ is called the \emph{index} of the DAE. Furthermore, we write $\C_+\coloneqq \{z\in\C~|~\re z>0\}$, $\C_-\coloneqq \{z\in\C~|~\re z<0\}$ and  $\overline{\C_-}\coloneqq \C\setminus\C_+$.

\section{Behavioral stable and dH DAEs}
\label{sec:stab}
In this section, we focus on DAEs
\begin{align}
\label{DAE_2}
\ddt Ex(t)=Ax(t),
\end{align}
with $E,A\in\R^{n\times n}$ and denote the set of all such pairs $[E,A]$ by $\Sigma_n$. Here we follow the behavioral approach from \cite{PoldWill97} and introduce the \emph{behavior} of $[E,A]$  by 
\begin{align*}
\mathfrak{B}_{[E,A]}^\infty\coloneqq\{x\in C^\infty(\R,\R^n) \, \big|  \tfrac{{\rm d}}{{\rm d} t}Ex=Ax\},
\end{align*}
A definition with less smoothness assumptions can be found in \cite{BergReis13}.

We say that $[E,A]\in\Sigma_n$ is \emph{stable} 
if 
\[
\forall x\in\mathfrak{B}_{[E,A]}^\infty\,\, \exists M\geq 0 :\quad  \sup_{t\geq 0}\|x(t)\|\leq M.
\]

The aim is to show that stable DAEs  $[E,A]$ can be rewritten as a dH DAE \eqref{ph_DAE_i}, i.e.,
\begin{align}
\label{pH_DAE}
\begin{split}
&\tfrac{\rm d }{{\rm d}t}Ex(t)=(J-R)Qx(t),\\ &J=-J^T,\quad  R\geq 0,\quad  Q^TE\geq 0. 
\end{split}
\end{align}

To this end, we introduce the \emph{system space} as follows
\begin{align*}
\Vs^{[E,A]}&\coloneqq \{x(0)~|~ x\in \mathfrak{B}_{[E,A]}^\infty\}\subseteq\R^{n},
\end{align*}
see also \cite{ReisRend15} for an equivalent definition. 
To determine the system space first observe that the behavior and hence the system space transforms as follows
\begin{align}
\label{EA_inv}
\mathfrak{B}_{[SET,SAT]}^\infty&=T^{-1}\mathfrak{B}_{[E,A]}^\infty,\quad \Vs^{[SET,SAT]}&=T^{-1}\Vs^{[E,A]}
\end{align}
for invertible $S,T\in\R^{n\times n}$. This combined with the quasi Kronecker form \eqref{kcf} is the following representation of the system space.
\begin{lem}
\label{lem:sys_sp}
Let $[E,A]\in\Sigma_{n}$ with invertible $S,T\in\R^{n\times n}$ such that $S(sE-A)T$ is in quasi Kronecker form \eqref{kcf}. Then 
\begin{align}
\label{kron_vsys}
\Vs^{[E,A]}=T(\R^{n_0}\times \{0\}^{|\alpha|}\times \R^{|\beta|}\times \{0\}^{|\gamma|-\ell_\gamma}).
\end{align}
\end{lem}

Below, the stability is characterized either in terms of the eigenvalues or a Lyapunov inequality on the system space. To this end, for a subspace $\Lc\subseteq\R^n$ and matrices $M,N\in\R^{n\times n}$ we say that
\[M=_{\Lc}N\quad :\Leftrightarrow\quad Mx=Nx\quad \forall ~x\in\Lc.\]
If we further assume that $M$ and $N$ are symmetric, we say that
\begin{align*}
M\geq_{\Lc}N\quad &:\Leftrightarrow\quad x^TMx\geq x^TNx\quad \forall ~x\in\Lc,\\
M>_{\Lc}N\quad &:\Leftrightarrow\quad x^TMx>x^TNx\quad \forall~x\in\Lc\setminus\{0\}.
\end{align*}
\begin{prop}
\label{prop:fred}
Let $[E,A]\in\Sigma_n$ then the following statements are equivalent.
\begin{itemize}
\item[\rm (a)] $[E,A]$ is stable.
\item[\rm (b)] $sE-A$ is regular and there exists a symmetric $X\in\R^{n\times n}$ with $X>_{E\Vs^{[E,A]}}0$ and  $X(E\Vs^{[E,A]})=E\Vs^{[E,A]}$ such that
\begin{align}
\label{lyap}
A^TXE+E^TXA\leq_{\Vs^{[E,A]}}0.
\end{align}
\item[\rm (c)] The pencil $sE-A$ is regular,  $\sigma(E,A)\subseteq\overline{\C_-}$ and the eigenvalues on the imaginary axis are semi-simple. 
\end{itemize}
\end{prop}
\begin{proof}
Since the above conditions are invariant under multiplication from left and right with invertible matrices, we can assume that $sE-A$ is already given in Kronecker form \eqref{kcf}. If $[E,A]$ is stable then $\beta$ cannot be present in the Kronecker form. Indeed if $\beta_1\geq 1$ then for all $f\in C^\infty(\R,\R)$, there exists a solution $\tilde x$ to
\[\ddt K_{\beta_1}\tilde x(t)=L_{\beta_1}\tilde x(t),\quad \tilde x_{\beta_1}=f,\] 
which contradicts stability. Since $sE-A$ is square also $\gamma$ is not present and therefore $sE-A$ is regular. Furthermore, the stability of $[I_{n_0},A_0]$ shows $\sigma(E,A)\subseteq\overline{\C_-}$ and that all eigenvalues on the imaginary axis are semi-simple. This proves (c). Clearly, (c) implies (a). To prove the equivalence of (b) and (c), we assume that $sE-A$ is regular.
Then Lemma~\ref{lem:sys_sp} implies that 
\begin{align*}
\Vs^{[E,A]}=\R^{n_0}\times \{0\}^{|\alpha|},\quad E\Vs^{[E,A]}=\Vs^{[E,A]}.
\end{align*}
Furthermore, \eqref{lyap} holds if and only if 
\[
\begin{smallbmatrix}
A_0^T&0\\0&I_{|\alpha|}
\end{smallbmatrix}X\begin{smallbmatrix}
I_{n_0}&0\\0&N
\end{smallbmatrix}+\begin{smallbmatrix}
I_{n_0}&0\\0&N^T
\end{smallbmatrix}X\begin{smallbmatrix}
A_0&0\\0&I_{|\alpha|}
\end{smallbmatrix}\leq _{\R^{n_0}\times \{0\}^{|\alpha|}}0
\]
which is, by considering the upper-left block entries, equivalent to the existence of some $X_1\in\R^{n_0\times n_0}$ with $X_1>0$ such that 
\begin{align}
\label{red_lyap}
A_0^TX_1+X_1A_0\leq 0
\end{align}
and hence to $\sigma(A_0)\subseteq\overline{\C_-}$ and semi-simple eigenvalues on the imaginary axis. If $X_1>0$ is a solution of \eqref{red_lyap} then $X\coloneqq\begin{smallbmatrix}
X_1&0\\0&0
\end{smallbmatrix}$ fulfills $XE\Vs^{[E,A]}=E\Vs^{[E,A]}$ and $X_{E\Vs^{[E,A]}}>0$.
\end{proof}
\begin{rem}
Oftentimes \emph{asymptotically stable} DAEs are of interest which are defined by
\[
\forall x\in\mathfrak{B}_{[E,A]}^\infty: \lim\limits_{t\rightarrow\infty}\|x(t)\|=0.
\]
Analogously to Proposition~\ref{prop:fred}, such  DAEs can be characterized by a strict inequality in \eqref{lyap} or by $\sigma(E,A)\subseteq\C_-$. Related characterizations with Lyapunov equations were previously given e.g.\ in \cite{Styk02} and recently in  \cite[Theorem~4]{Mehr21}.
\end{rem}


The proposition below provides a dH formulation \eqref{pH_DAE} of stable DAEs $[E,A]\in\Sigma_n$ on the system space. Here we use the pseudo-inverse $Q^\dagger$ of a matrix $Q$.
\begin{prop}
\label{prop:pH}
Let $[E,A]\in\Sigma_n$ be stable and let $X\in\R^{n\times n}$ be a solution of \eqref{lyap}. Define $Q\coloneqq XE$ and  \[J\coloneqq\tfrac{1}{2}(AQ^{\dagger}-(AQ^\dagger)^T),\ \ R\coloneqq-\tfrac{1}{2}(AQ^\dagger+(AQ^\dagger)^T).\]
Then $[E,(J-R)Q]$ is a dH DAE on $\Vs^{[E,A]}$ in the sense that
\begin{equation*}
    \begin{split}
        J&=_{E\Vs^{[E,A]}}-J^T,\\A&=_{\Vs^{[E,A]}}(J-R)Q,
    \end{split}\quad\quad
    \begin{split}
    R&\geq_{E\Vs^{[E,A]}}0,\\Q^TE&>_{\Vs^{[E,A]}}0.
    \end{split}
\end{equation*}

\end{prop}
\begin{proof}
Note that $sE-A$ is regular by Proposition \ref{prop:fred}. Let $X\in\R^{n\times n}$ be a solution to \eqref{lyap} as given by Proposition \ref{prop:fred}. Since $XE\Vs^ {[E,A]}=E\Vs^{[E,A]}=\Vs^{[E,A]}$ by regularity of $sE-A$ and \eqref{kron_vsys}, 
\[Q^\dagger Q=_{\Vs^{[E,A]}}E^\dagger E=_{\Vs^{[E,A]}}I=_{\Vs^{[E,A]}}EE^\dagger=_{\Vs^{[E,A]}}QQ^\dagger\] 
and thus $(J-R)Q=AQ^\dagger Q=_{\Vs^{[E,A]}}A$.
Since $X>_{E\Vs^{[E,A]}}0$, we have for all $x\in \Vs^{[E,A]}\setminus\{0\}$
\[
x^TQ^TEx=x^TE^TXEx>0.
\]
With $E\Vs^{[E,A]}=Q\Vs^{[E,A]}$ we find that $R\geq_{E\Vs^{[E,A]}}0$ is equivalent to 
\begin{align*}
0\leq_{\Vs^{[E,A]}} Q^TRQ&=-Q^T(AQ^\dagger+(Q^\dagger)^TA^T)Q\\&=-Q^TA-A^TQ\\&=_{\Vs^{[E,A]}}-E^TXA-A^TXE,
\end{align*}
which holds by \eqref{lyap}. Moreover, $J=_{E\Vs^{[E,A]}}J^T$ holds trivially.
\end{proof}
\begin{rem}
\label{rem:index_one}
If $[E,A]\in\Sigma_n$ is stable with index at most one then we can redefine $Q$ in Proposition~\ref{prop:pH} in such a way that the relations in (b) hold on $\R^n$. 
Here we assume for simplicity that $E$ and $A$ are already given in quasi Kronecker form.  Then $E=\begin{smallbmatrix}
I_{n_0}&0\\0&0
\end{smallbmatrix}$ and $A=\begin{smallbmatrix} A_0&0\\0&I_{n-n_0}
\end{smallbmatrix}$. Let $X_0>0$ satisfy  $A_0^TX_0+X_0A_0\leq 0$. Then we define $\hat Q\coloneqq \begin{smallbmatrix} X_0&0\\0&-I_{n-n_0}
\end{smallbmatrix}$ which is invertible and with $J\coloneqq\tfrac{1}{2}(A\hat Q^{-1}-(A\hat Q^{-1})^T)$ and $R\coloneqq-\tfrac{1}{2}(A\hat Q^{-1}+(A\hat Q^{-1})^T)$ we have $\hat Q^TE\geq 0$ and $A=(J-R)\hat Q$. If $sE-A$ has index greater than one then this extension of $Q$ is still possible but leads to $\hat Q^TE\ngeq0$.
\end{rem}



The following example from \cite{MehlMehrWojt18} shows that not every dH DAE given by \eqref{pH_DAE} is stable. Consider
\[
sE-JQ=\begin{smallbmatrix}
s&1\\ 0&s
\end{smallbmatrix},\quad  J\coloneqq \begin{smallbmatrix}0&-1\\ 1&0\end{smallbmatrix},\quad  Q\coloneqq \begin{smallbmatrix}0&0\\ 0&1\end{smallbmatrix},
\]
which is port-Hamiltonian and has a Jordan block of size 2 at zero and is therefore unstable. The above example is an ordinary differential equation and $x\mapsto x^TQ^TEx=x^TQx$ is not a Lyapunov function.

It will be shown below, that an additional assumption which guarantees the stability of dH DAEs is 
\begin{align}
\label{ker_eq}
\ker Q\subseteq\ker E.
\end{align}
An interpretation of this condition is given later in the geometric formulation of DAEs in Section~\ref{sec:geo} where it ensures that the only trajectories of a system with vanishing effort are the purely algebraic solutions.

\begin{prop}
\label{prop:postfred}
Let $[E,(J-R)Q]$ be a dH DAE \eqref{pH_DAE} with $\ker Q\subseteq\ker E$. Then the following holds: 
\begin{itemize}
\item[\rm (a)] If $sE-Q$ is regular then
$Q$ is invertible. 
\item[\rm (b)] Let $Q$ be invertible. Then the dH DAE 
\eqref{pH_DAE} is stable if and only if $\ker J\cap\ker R\cap (Q\ker E)=\{0\}$. 
\end{itemize}
\end{prop}
\begin{proof} 
If $sE-Q$ is regular then $\ker E\cap\ker Q=\{0\}$. This together with $\ker Q\subseteq\ker E$ implies $\ker Q=\{0\}$ and hence $Q$ is invertible. This proves (a).

Next, observe that 
\begin{align}
\label{ker_schnitt}
\ker(J-R)=\ker J\cap\ker R. 
\end{align}
If $x\in \ker(J-R)$ then $x^T(J-R)x=0$ and this implies $0=x^T(J^T-R)x=-x^T(J+R)x$. Hence $x\in\ker R$ and therefore $x\in\ker J$. The reverse inclusion is trivial.

To characterize the stability of \eqref{pH_DAE}, we use Proposition~\ref{prop:fred}~(c). 
If \eqref{pH_DAE} is stable then $sE-(J-R)Q$ is regular. Hence, using the invertibility of $Q$ and \eqref{ker_schnitt} we find
\begin{align}
    \label{ker_schnitt_3}
    \{0\}=\ker E\cap\ker (J-R)Q=(Q\ker E)\cap\ker J\cap\ker R.
\end{align} 
Conversely, assume that \eqref{ker_schnitt_3} holds. The pencil $sEQ^{-1}-(J-R)$ is \emph{semi-dissipative Hamiltonian} and according to \cite[Thm.~2]{MehlMehrWojt20} in the quasi-Kronecker form \eqref{kcf} one has $\beta_i,\gamma_i\leq 1$ and the eigenvalue conditions in Proposition~\ref{prop:fred}~(c) are satisfied for the regular part \eqref{kcf_a}. Since 
\[
\{0\}=\ker E\cap\ker (J-R)Q=\ker EQ^{-1}\cap\ker(J-R),
\]
the multi-index $\beta$ cannot be present in the quasi-Kronecker form of $sEQ^{-1}-(J-R)$. Hence $\gamma$ is not present implying that $sE-(J-R)Q$ is regular and therefore stable.
\end{proof}

\begin{rem}
Similar considerations can be done for systems which are stable \emph{backwards in time}, i.e., $[-E,A]$ is stable.
A characterization similar to Proposition~\ref{prop:fred} is straightforward and as a consequence, every DAE which has semi-simple eigenvalues on the imaginary axis can be rewritten on the system space $\Vs^{[E,A]}$ in the form \eqref{pH_DAE} but with $Q^TE=E^TQ$ instead of $Q^TE\geq 0$.
\end{rem}

\section{Stabilizable systems as pH systems}
\label{sec:stabi}
In this section, we consider descriptor systems 
\[
\ddt Ex(t)=Ax(t)+Bu(t),
\]
with $E,A\in\R^{n\times n}$, $B\in\R^{n\times k}$ for $n,k\in\N$ and we write  $[E,A,B]\in\Sigma_{n,k}$. The \emph{behavior} and \emph{system space}  of $[E,A,B]$ are given by 
\begin{multline*}
\mathfrak{B}_{[E,A,B]}^\infty\coloneqq\{(x,u)\in C^\infty(\R,\R^n)\times C^\infty(\R,\R^k) \, \big|\\ \tfrac{{\rm d}}{{\rm d} t}Ex=Ax+Bu\},\\
\Vs^{[E,A,B]}\coloneqq \{(x(0),u(0))~|~ (x,u)\in \mathfrak{B}_{[E,A,B]}^\infty\}\subseteq \R^{n+k}.
\end{multline*}

In the following it is shown that systems
which can be stabilized in behavioral sense can be rewritten as pH systems \eqref{ph_sys} on the system space after introducing a suitable output. 

We say that the system $[E,A,B]\in\Sigma_{n,k}$ is \emph{behaviorally stabilizable} if 
\begin{align}
\label{weaker}
\begin{split}
\forall x_0\in\Vs^{[E,A]}\,\, ~\exists M\geq 0~ \exists  (x,u)\in\mathfrak{B}_{[E,A,B]}^\infty:\\ x(0)=x_0,\quad \sup\limits_{t\geq 0}\|x(t)\|\leq M.
\end{split}
\end{align}
Furthermore, the  class of behavioral stabilizable systems $[E,A,B]\in\Sigma_{n,k}$ for which $sE-A$ is regular and only has semi-simple eigenvalues on the imaginary axis are denoted by $\Sigma_{n,k}^s$.

If $[E,A,B]\in\Sigma_{n,k}^s$ then applying a Jordan decomposition on the first block entry in the quasi Kronecker form \eqref{kcf} we see that  there exists some invertible $S,T\in\R^{n\times n}$, $n_1,n_2\in\N$ with $n_1+n_2=n_0$ such that
\begin{align}
\label{weier_refined}
\begin{split}
&S(sE-A)T=\begin{smallbmatrix}
sI_{n_1}-A_1&0&0\\0&sI_{n_2}-A_2&0\\0&0&sN_\alpha-I_{|\alpha|}
\end{smallbmatrix},\quad x=\begin{smallpmatrix}
x_1\\x_2\\x_3
\end{smallpmatrix} \\ &SB=\begin{smallbmatrix}
B_1\\B_2\\B_\alpha
\end{smallbmatrix} ,\quad \sigma(A_1)\subseteq\C_+,\quad [I_{n_2},A_2]~\text{stable.} 
\end{split}
\end{align}
Furthermore, $[I_{n_1},A_1,B_1]$ is controllable.
\begin{lem}
\label{lem:diag}
If $[E,A,B]\in\Sigma_{n,k}^s$ fulfills \eqref{weier_refined} with $S=T=I_n$ then a stabilizing feedback is given by $u(t)=-B_1^TP_1x_1(t)$ and $P_1>0$ is the unique solution of 
    \begin{align}
\label{ABE}
    A_1^TP_1+P_1A_1=P_1B_1B_1^TP_1.
    \end{align}
\end{lem}
\begin{proof}
Clearly, \eqref{ABE} has a unique positive definite solution if and only if 
\[
    (-A_1^T)^TP_1^{-1}+P_1^{-1}(-A_1^T)=-B_1B_1^T
\]

has a unique positive definite solution. Since $[I_{n_1},A_1,B_1]$ is controllable and $\sigma(-A_1^T)\subset\C_-$ such a solution exists by \cite[Thm.~3.28]{TrenStooHaut01}. It remains to show that $u(t)=-B_1^TP_1x_1(t)$ is a stabilizing feedback, i.e., that \eqref{weaker} holds. Note that $[I_{n_1},A_1-B_1B_1^ TP_1,B_1]$ is also controllable and that \eqref{ABE} is also equivalent to
\[
    (A_1^T-P_1^TB_1B_1^T)^TP_1^{-1}+P_1^{-1}(A_1^T-P_1^TB_1B_1^T)=-B_1B_1^T.
\]
Now invoking \cite[Thm.~3.28]{TrenStooHaut01} again shows $\sigma(A_1-B_1B_1^TP_1)\subseteq\C_-$. Moreover, with
\[
x_1(t)=e^{(A_1-B_1B_1^TP_1)t}x_1(0),\quad \forall t\geq 0,
\]
there exist $M,\beta>0$ such that for all $k\geq 0$ which are smaller than the index of $[E,A]$
\begin{align*}
\|x_1^{(k)}(t)\|\leq Me^{-\beta t},\quad \forall t\geq 0.
\end{align*}
Therefore, we have for some $\hat M>0$ such that for all $k\geq 0$ which are smaller then the index of $[E,A]$ 
\begin{align}
\label{all_u_estim}
\|u^{(k)}(t)\|\leq \hat Me^{-\beta t},\quad \forall t\geq 0.
\end{align}
Since $[I_{n_2},A_2]$ is stable, the variation of constants formula implies that the solution $x_2$ of $\dot x_2=A_2x_2(t)+B_2u(t)$ are bounded.  Next, we only consider the first block-entry (if any) of $sN_\alpha-I_{|\alpha|}$, $sN_{\alpha_1}-I_{\alpha_1}$, since the others are treated analogously. With $B_\alpha=\begin{smallbmatrix}
B_{\alpha_1}^T&\cdots&B^T_{\alpha_{\ell(\alpha)}}
\end{smallbmatrix}^T$ the solution $x_{\alpha_1}(t)=(x_{{\alpha_1},1}(t),\ldots,x_{{\alpha_1},{\alpha_1}}(t))^T$  of $[N_{\alpha_1},I_{\alpha_1}]$ fulfills 
\[
\begin{smallpmatrix}
\dot x_{{\alpha_1},2}\\ \vdots\\ \dot x_{{\alpha_1},{\alpha_1}}\\ 0
\end{smallpmatrix}=\ddt N_{\alpha_1}\begin{smallpmatrix}
x_{{\alpha_1},1}\\ \vdots\\ x_{{\alpha_1},{\alpha_1}-1}\\ x_{{\alpha_1},n_{\alpha_1}}
\end{smallpmatrix}=\begin{smallpmatrix}
x_{{\alpha_1},1}\\ \vdots\\ x_{{\alpha_1},{\alpha_1}-1}\\ x_{{\alpha_1},n_{\alpha_1}}
\end{smallpmatrix}+B_{\alpha_1}u
\]
and inspecting the last row leads to
\[
\|x_{{\alpha_1},{\alpha_1}}(t)\|=\|{\alpha_1}u(t)\|\leq \|B_{\alpha_1}\|\hat Me^{-\beta t},
\]
which tends to zero as $t\rightarrow\infty$. Similarly, the  penultimate row leads to
\begin{align*}
x_{{\alpha_1},{\alpha_1}-1}(t)&=\dot x_{{\alpha_1},{\alpha_1}}(t)-e_{{\alpha_1}-1}^TB_{\alpha_1}u(t)\\
&=B_{\alpha_1}\dot u(t)-e_{{\alpha_1}-1}^TB_{\alpha_1}u(t)
,\end{align*}
which is again exponentially bounded  by \eqref{all_u_estim}.
Repeating the last step with the remaining rows starting with the ${\alpha_1}-2$th row shows that $\|x_{\alpha_1}(t)\|\rightarrow 0$ as $t\rightarrow\infty$ which completes the proof that $u$ stabilizes the solution of $[E,A,B]$. 
\end{proof}

Below, we obtain another characterization of stabilizability using solutions to certain matrix equalities on the system space which allows us to reformulate the system in a port-Hamiltonian way.
Similar equations for stabilization of DAEs have been studied  under the name \emph{generalized algebraic Bernoulli equation} in \cite{BarrBennQuin07}.

\begin{lem}\label{prop:GIRLS-reduced}
Let $[E,A,B]\in\Sigma_{n,k}^s$. Then there exists some $X_1,X_2\geq_{E\Vs^{[E,A]}}0$ such that the following holds.
\begin{align}
        &A^TX_1E+E^TX_1A=_{\Vs^{[E,A]}}E^TX_1BB^TX_1E,\label{GABd1}\\
        &A^TX_2E+E^TX_2A\leq_{\Vs^{[E,A]}}0,\label{GABd2}
        \end{align}
with $(X_2\pm X_1)E\Vs^{[E,A]}=E\Vs^{[E,A]}$ and $X_1+X_2>_{E\Vs^{[E,A]}}0$.
\end{lem}
\begin{proof}
All conditions are invariant under transformations of the form $[E,A,B]\rightarrow[SET,SAT,SB]$ for all invertible $S,T\in\R^{n\times n}$. Hence we can assume without restriction that $[E,A,B]$ is already given in the block diagonal form on the right-hand side of \eqref{weier_refined}. Introduce $\tilde E=\diag(I_{n_2},N)$, $\tilde A=\diag(A_2,I_{n_3})$ and then we set $X_1\coloneqq \diag(P_1,0_{n_2+n_3})$, where $P_1>0$ is a solution of \eqref{ABE} and $X_2\coloneqq \diag(0_{n_1},P_2^{-1},0_{n_3})$ where $P_2>0$ is a solution of the Lyapunov inequality $ A_2^TP_2+P_2A_2\leq0$. Clearly, $X_1,X_2\geq_{E\Vs^{[E,A]}}0$ and they satisfy \eqref{GABd1} and \eqref{GABd2}, respectively. 
Furthermore,  $E\Vs^{[E,A]}=\R^{n_1+n_2}\times \{0\}^{n_3}$ and hence $X_2\pm X_1=\diag(\pm P_1,P_2^{-1},0_{n_3})$ map $E\Vs^{[E,A]}$ into itself with $X_1+X_2>_{E\Vs^{[E,A]}}0$.
\end{proof}

Based on this result, we show how to interpret a stabilizable system as a port-Hamiltonian system \eqref{ph_sys} which can be viewed as an analogue to Proposition~\ref{prop:pH}. 

\begin{prop}\label{cor:feedback-pH}
Let $[E,A,B]\in\Sigma_{n,k}^s$. Then there exist $X_1,X_2\geq_{E\Vs^{[E,A]}}0$ such that with the choices of
    \begin{equation}\label{eq:choice}
        \begin{aligned}
        Q&\coloneqq(X_2-X_1)E,\\
        J&\coloneqq\tfrac{1}{2}(AQ^{\dagger}-(AQ^{\dagger})^T),\\
        R&\coloneqq-\tfrac{1}{2}(AQ^{\dagger}+(AQ^{\dagger})^T),
        \end{aligned}
    \end{equation}
the system
\begin{equation}\label{eq:feedback-pH}
\begin{aligned}
    \tfrac{\rm d}{{\rm d} t}  Ex(t)&=(J-R)Qx(t)+Bu(t),\\
    y(t)&=B^TQx(t)+u(t),
\end{aligned}
\end{equation}
is port-Hamiltonian on $\Vs^{[E,A]}\times\R^k$ in the sense that
\begin{equation}\label{eq:pH-Vsys}
\begin{aligned}
\begin{smallbmatrix}
J&B\\-B^T&0
\end{smallbmatrix}&=_{E\Vs^{[E,A]}\times\R^k}-\begin{smallbmatrix}
J&B\\-B^T&0
\end{smallbmatrix}^T,\\
\begin{smallbmatrix}
R&0\\0&I
\end{smallbmatrix}&\geq_{E\Vs^{[E,A]}\times\R^k}0,\\
E^TQ&=_{\Vs^{[E,A]}}Q^TE,\quad A=_{\Vs^{[E,A]}}(J-R)Q.
\end{aligned}
\end{equation}
\end{prop}

\begin{proof}
The proof is completely analogous to the proof of Proposition~\ref{prop:pH} except that we use \eqref{GABd1} and \eqref{GABd2} instead of \eqref{lyap} to prove the inequality in \eqref{eq:pH-Vsys}. To be more precise, let $X_1,X_2$ be the solutions of \eqref{GABd1} and \eqref{GABd2}, respectively, given by Lemma \ref{prop:GIRLS-reduced}. Then with $(X_1-X_2)E\Vs^{[E,A]}=E\Vs^{[E,A]}=\Vs^{[E,A]}$ we have 
\[R\geq_{E\Vs^{[E,A]}}0\Leftrightarrow 2Q^TRQ\geq_{\Vs^{[E,A]}}0\]
and
\begin{align*}
    2Q^TRQ&=_{\Vs^{[E,A]}}-Q^TA-A^TQ\\&=-E^T(X_2-X_1)A-A^T(X_2-X_1)E\\
    &\geq_{\Vs^{[E,A]}}E^TX_1BB^TX_1E\geq_{\Vs^{[E,A]}}0.
\end{align*}
\end{proof}

\begin{rem}
In the context of Proposition~\ref{cor:feedback-pH}, one can show that the solutions $(x,u)$ of  $[E,A-BB^T(X_2-X_1)E]$ correspond to the solutions $(x,u,y)$ of \eqref{eq:feedback-pH} when imposing $y=0$. This restriction corresponds to an \emph{interconnection} (cf. \cite{CervvdScBano07}) with respect to the \emph{Dirac structure} $\ran\begin{smallbmatrix}
0\\I
\end{smallbmatrix}$ which are key objects in the geometric formulation of pH systems (cf.\ Section~\ref{sec:geo}).
Furthermore, if $sE-A$ has index one then $N=0$ in \eqref{weier_refined} and $E\Vs^{[E,A]}\times\R^k$ coincides with $\begin{smallbmatrix}
E&0\\0&I_k
\end{smallbmatrix}\Vs^{[E,A,B]}$. If $E$ is invertible then $\Vs^{[E,A,B]}=\R^{n+k}$,  $E\Vs^{[E,A]}=\R^n$ and $Q=X_1+X_2>0$.
\end{rem}
\begin{rem}
In Proposition~\ref{cor:feedback-pH} we defined a suitable output to obtain a pH system. More generally, in \cite[Thm.~3.6]{GillShar18} it was shown that descriptor systems $[E,A,B]$ with output $y(t)=Cx(t)+Du(t)$ can be rewritten as a pH system of the form \eqref{ph_sys} if there exists an invertible solution $X$ of a linear matrix inequality (typically referred to as Kalman-Yakubovich-Popov inequality). Moreover, one has $X=Q$ in \eqref{ph_sys}. 
Such invertible solutions can be obtained by restricting the system \eqref{eq:feedback-pH} to the space $\Vs^{[E,A]}\times \R^k$. In this case $Q$ in \eqref{eq:choice} can be replaced by the invertible $\hat Q\coloneqq (X_2-X_1)E|_{\Vs^{[E,A]}}$.
\end{rem}

\section{Stability of geometric dH DAEs}
\label{sec:geo}
In this section we study the stability of DAEs which are given by the geometrical formulation used in \cite{MascvdSc18}.

In comparison to \eqref{ph_sys}, the equations here are given implicitly. To this end, let $\mathcal{L}$ and $\mathcal{D}$ be subspaces of $\R^{n}\times \R^{n}$ with \emph{range representations} $\Lc=\ran\begin{smallbmatrix}
L_1\\L_2
\end{smallbmatrix}$ and $\Dc=\ran\begin{smallbmatrix}
D_1\\D_2
\end{smallbmatrix}$ for some  $L_1,L_2,D_1,D_2\in\R^{n\times n}$. Then $\Lc$ is \emph{Lagrangian} and $\Dc$ is \emph{maximally dissipative} if
\begin{align}
\label{lag_dissip}
L_1^TL_2-L_2^TL_1&=0,& &\dim\Lc=n, \quad\text{and}\\ D_2^TD_1+D_1^TD_2&\leq 0,&&\dim\Dc=n. \label{nur_dissip}
\end{align}
Furthermore, the matrices $L_1,L_2,D_1,D_2$ can be chosen in such a way that $L_2=L_2^T$ and $D_2+D_2^T\leq 0$ holds. These $\mathcal{D}$ and $\mathcal{L}$ can be used to implicitly define a DAE by demanding that the system trajectories $z,e:\R\rightarrow\R^n$ satisfy 
\begin{align}
\label{pH_geo}
(e(t),-\tfrac{d}{dt}z(t))\in\mathcal{D}, \quad (z(t),e(t))\in\mathcal{L}.
\end{align}
These systems were called \emph{generalized pH DAE systems} in \cite{MascvdSc18} and $z$ and $e$ are called the \emph{state} and \emph{effort}, respectively. Therein, $\Dc$ was assumed to be a so-called \emph{Dirac structure}, i.e.,\ \eqref{nur_dissip} holds with equality. In  \cite{GernHall21}, $\Dc$ in \eqref{pH_geo} is allowed to be dissipative and we include the port and the resistive variables already in the state for simplicity. 

The DAE is then explicitly given by the range representation of the following subspace
\begin{align}
\Dc\Lc&\coloneqq \{(x,z) ~|~ \exists\,e\in\R^n:~(x,e)\in\Lc, (e,z)\in \Dc\}\nonumber \\&~=\ran\begin{bmatrix}
E\\A
\end{bmatrix} \label{def_pencil}
\end{align}
for some $E,A\in\R^{n\times n}$. Moreover, the functions which fulfill $\tfrac{d}{dt}Ex(t)=Ax(t)$ and \eqref{pH_geo} with $z(t)=Ex(t)$ coincide, see \cite[Section~4]{GernHall21}.


To ensure stability of $[E,A]\in\Sigma_n$ given by \eqref{def_pencil}, we need the additional assumption that the Lagrangian subspace $\Lc=\begin{smallbmatrix}
L_1\\ L_2
\end{smallbmatrix}$ is \emph{nonnegative}, i.e., $L_1^TL_2\geq 0$  and, using a suitable CS-decomposition, see \cite[Prop.~3.1]{MehlMehrWojt18}, we can choose $L_1$ and $L_2$ in such a way that $L_1\geq 0$. 

As a main result of this section we provide sufficient conditions for the stability of DAEs given by \eqref{def_pencil}. Here we use the orthogonal projector $P_{\mathcal{M}}$ onto a subspace $\mathcal{M}\subseteq\R^n$. 
The result is an immediate consequence of \cite[Prop.~5.1, Prop.~6.3]{GernHall21}.
\begin{prop}
\label{prop:stab_geo}
Let $\Dc=\ran\begin{smallbmatrix} D_1\\ D_2
\end{smallbmatrix}$ be maximally dissipative and $\Lc=\ran\begin{smallbmatrix}
L_1\\ L_2
\end{smallbmatrix}$ nonnegative Lagrangian with $L_1,L_1^TL_2\geq 0$ and $D_2+D_2^T\leq 0$. Further, let $sE-A$ be the pencil given by \eqref{def_pencil}. Then with $\mathcal X\coloneqq \ran D_1\cap\ran L_2$  the following holds.  
\begin{itemize}
    \item[\rm (a)] $sE-A$ is regular if and only if 
    \begin{align*}
 \ker P_{\mathcal X}L_1|_{\mathcal X}\cap\ker P_{\mathcal X}D_{2}|_{\mathcal X}=\{0\}, \\   
D_2(\ker D_1)\cap L_1(\ker L_2)=\{0\}.
\end{align*}
    \item[\rm (b)] $[E,A]$ is stable if additionally $L_1(\ker L_2)=\{0\}$ holds.
\end{itemize}
\end{prop}


DH DAEs $sE-(J-R)Q$ can be written in the form \eqref{def_pencil} using $\Dc={\rm gr\,} (J-R)\coloneqq \{(x,(J-R)x) ~|~x\in\R^n\}$ and $\Lc=\ran\begin{smallbmatrix}
E\\Q
\end{smallbmatrix}$. Here, $D_1=I_n$ and $D_2=J-R$ and hence the second condition in Proposition~\ref{prop:stab_geo}~(a) trivially holds.
\begin{corollary}
\label{cor:stable}
Let $[E,(J-R)Q]\in\Sigma_n$ be a dH DAE \eqref{ph_DAE_i} then it is regular if and only if $\ker E\cap\ker(Q^TJQ)\cap\ker(Q^TRQ)=\{0\}$. Furthermore, $[E,(J-R)Q]$ is stable if additionally $\ker Q\subseteq \ker E$.
\end{corollary}
\begin{proof}
Since stability and regularity are invariant under pencil equivalence we can use 
\cite[Prop. 3.1]{MehlMehrWojt18} and assume without restriction that $E\geq 0$, $Q=Q^T\geq 0$.
If $sE-(J-R)Q$ is regular then $\ker E\cap\ker Q=\{0\}$ and hence \cite[Prop. 3.1]{MehlMehrWojt18} implies that $\Lc=\ran\begin{smallbmatrix}
E\\Q
\end{smallbmatrix}$ is Lagrangian and without restriction $E\geq 0$, $Q^TE\geq 0$, $Q=Q^T=Q^2$ holds. By Proposition~\ref{prop:stab_geo} with $\mathcal X=\ran Q=(\ker Q^T)^\perp$ the regularity is equivalent to
\begin{align}
&~~~~Q(\ker Q^TEQ\cap\ker(Q^TJQ)\cap\ker Q^TRQ) \label{mitohneprojektor}\\
&=\{Qx : Q^TEQx=Q^T(J-R)Qx=0\}\nonumber\\
&=\ker P_{\ran Q}E|_{\ran Q}\cap\ker P_{\ran  Q}(J-R)|_{\ran Q}\nonumber\\
&=\{0\}.\nonumber
\end{align}
The assumptions on $E$ and $Q$ imply
\[
\ker Q\dotplus \ker E=\ker QE=\ker EQ=\ker QEQ
\]
and therefore \eqref{mitohneprojektor} is equivalent to
\begin{align*}
&~~~~~~~Q(\ker QEQ\cap\ker Q(J-R)Q)=\{0\}\\
&\Longleftrightarrow (\ker E\dotplus \ker Q)\cap\ker Q(J-R)Q\subseteq\ker Q\\
    &\Longleftrightarrow\ker E\cap \ker Q(J-R)Q=\{0\}.
\end{align*}
Furthermore, \eqref{ker_schnitt} implies
\[
\ker Q(J-R)Q=\ker(QJQ)\cap\ker(QRQ)
\]
which proves $\ker E\cap \ker Q^TJQ\cap \ker Q^TRQ=\{0\}$. Conversely, if this intersection is trivial then $\ker E\cap\ker Q=\{0\}$ and hence by \cite[Prop. 3.1]{MehlMehrWojt18} the subspace $\Lc=\ran\begin{smallbmatrix}
E\\Q
\end{smallbmatrix}$ fulfills $\dim\Lc=n$ and is therefore  Lagrangian. Hence \eqref{mitohneprojektor} holds and $sE-(J-R)Q$ is regular by Proposition~\ref{prop:stab_geo}. Conversely, if $\ker E\cap\ker(Q^T(J-R)Q)=\{0\}$ holds then $\ker Q\cap\ker E=\{0\}$ and hence $\Lc=\ran\begin{smallbmatrix}
E\\Q
\end{smallbmatrix}$ is Lagrangian. Hence Proposition~\ref{prop:stab_geo} implies that $sE-(J-R)Q$ is regular. To prove stability we apply Proposition~\ref{prop:stab_geo}~(b) and observe that
$E\ker Q=\{0\}$ is equivalent to \eqref{ker_eq}, i.e., $\ker Q\subseteq \ker E$.
\end{proof}
The above characterization of regularity was also obtained recently in \cite{Fritzetal21}. However,  Proposition~\ref{prop:stab_geo} characterizes the regularity and the stability for a larger class of DAEs. If $Q$ is invertible then the assumptions in Proposition~\ref{prop:stab_geo} and  Proposition~\ref{prop:postfred} (b) coincide.    



In the remainder we consider pH descriptor systems \eqref{ph_sys} and formulate them implicitly similar to  \eqref{pH_geo}. Define 
\[
\mathcal{D}\coloneqq{\rm gr\,} D={\rm gr\,}\begin{smallbmatrix}
J-R& B-P\\(B+P)^T& S-N
\end{smallbmatrix},\quad \mathcal{L}\coloneqq\ran \begin{smallbmatrix}
E&0\\ 0& I_k\\ Q&0 \\ 0& I_k
\end{smallbmatrix}.
\]
It follows from \eqref{ph_cond} that $D+D^T\leq 0$ and, hence, $\mathcal{D}$ fulfills \eqref{nur_dissip}. Moreover, $\mathcal{L}$ is Lagrangian if and only if $Q^TE=E^TQ$ and $sE-Q$ is regular, see \cite[Corollary 5.1]{GernHall21}. 

The system \eqref{ph_sys} is then implicitly given by 
\[
(z(t),u(t),-\ddt z(t), -y(t))\in\mathcal{D}\mathcal{L}=\ran\begin{smallbmatrix}
\hat E\\ \hat A
\end{smallbmatrix},
\]
with $\hat E\coloneqq \begin{smallbmatrix}E&0\\0&I_k
\end{smallbmatrix}$, $\hat A\coloneqq D\begin{smallbmatrix}Q&0\\0&I_k
\end{smallbmatrix}$. Hence, by Proposition~\ref{cor:feedback-pH} also stabilizable descriptor systems lie in the geometric framework of \cite{MascvdSc18}.

\section{Conclusion}
We studied the stability of dH DAEs \eqref{ph_DAE_i} and DAEs given by the implicit pH formulation \eqref{def_pencil}. In both cases we obtained that regularity together with the assumption $\ker Q\subseteq \ker E$ guarantees the stability of theses DAEs. Furthermore, we showed in Proposition~\ref{prop:pH} that stable DAEs can always be reformulated as dH DAEs  if we restrict the coefficients of the underlying equations to the system space. This restriction is not necessary if we consider index one DAEs, see Remark~\ref{rem:index_one}, and the matrix $Q$ in \eqref{ph_DAE_i} can be obtained from the solution of the generalized Lyapunov inequality \eqref{lyap}. 
Similarly, we showed in Proposition~\ref{cor:feedback-pH} that stabiliazble systems can be reformulated as pH systems \eqref{ph_sys} using the solutions of generalized algebraic Bernoulli equations \eqref{GABd1} and \eqref{GABd2}. Finally, it was shown how dH DAEs \eqref{ph_DAE_i} and pH systems \eqref{ph_sys} can be embedded into the geometric pH framework from \cite{MascvdSc18}.

\subsection*{Acknowledgements}
The authors are indebted to Timo Reis for his valuable remarks and suggestions on an earlier draft of this manuscript. Furthermore, the authors would like to thank Volker Mehrmann for providing valuable references.

\bibliography{GH_arxiv}
\bibliographystyle{alpha}

\end{document}